\documentclass[a4paper,11pt,fleqn]{article}
\usepackage[obeyspaces,hyphens,spaces]{url}
\usepackage[dvipsnames]{xcolor}
\usepackage{amsmath}
\usepackage{amssymb}
\usepackage{amsfonts}
\usepackage{charter}
\usepackage{mathtools}
\usepackage{euscript}
\usepackage{pifont}     
\usepackage{srcltx}     
\usepackage{etoolbox}
\usepackage[scr=rsfs]{mathalfa}
\allowdisplaybreaks 
\definecolor{labelkey}{HTML}{0455BF}
\definecolor{refkey}{rgb}{0,0.6,0.0}
\definecolor{dblue}{HTML}{044EAF}
\definecolor{dgreen}{HTML}{02724A}
\definecolor{myellow}{HTML}{D97904}
\definecolor{dred}{HTML}{D90404}
\usepackage{upref}
\usepackage{hyperref}
\hypersetup{%
pdfauthor={Minh N. B\`ui},
colorlinks=true,linkcolor=dblue,citecolor=dgreen,urlcolor=dred}
\renewcommand\familydefault{\rmdefault}
\DeclareMathAlphabet{\mathrm}{OT1}{\familydefault}{m}{n}
\makeatletter
\def\operator@font{\mathgroup\symoperators\rm}
\makeatother
\makeatletter
\def\@mathmeasure#1#2#3{\setbox#1\hbox{%
  \m@th$#2#3$}}
\makeatother
\renewcommand{\leq}{\ensuremath{\leqslant}}
\renewcommand{\geq}{\ensuremath{\geqslant}}

\newcommand{\scal}[2]{\langle{{#1}\mid{#2}}\rangle}
\newcommand{\sscal}[2]{\big\langle{{#1}\mid{#2}}\big\rangle}

\newcommand{\menge}[2]{\big\{{#1}~|~{#2}\big\}} 
 
\newcommand{\GGG}{\ensuremath{\boldsymbol{\mathcal{G}}}}
\newcommand{\HHH}{\ensuremath{\boldsymbol{\mathcal{H}}}}

\newcommand{\SHH}{\ensuremath{\boldsymbol{\mathsf{H}}}}

\newcommand{\HH}{\mathcal{H}}

\newcommand{\GG}{\mathcal{G}}

\newcommand{\NN}{\mathbb{N}}
\newcommand{\pinf}{\ensuremath{{{+}\infty}}}

\newcommand{\zeroun}{\ensuremath{\left]0,1\right[}}

\newcommand{\RPP}{\ensuremath{\left]0,{+}\infty\right[}}

\newcommand{\emp}{\varnothing}
\newcommand{\Sum}{\displaystyle\sum}

\newcommand{\Id}{\mathrm{Id}}

\newcommand{\exi}{\ensuremath{\exists\,}}

\DeclareMathOperator{\zer}{zer}

\newcommand*\mute{{\mkern 2mu\cdot\mkern 2mu}}

\def\abstract{\noindent{\bfseries Abstract}. \ignorespaces}

\usepackage{theorem}
\newtheorem{theorem}{Theorem}

\newtheorem{corollary}[theorem]{Corollary}

\theoremstyle{plain}{\theorembodyfont{\rmfamily}%
}
\newtheorem{fact}[theorem]{Fact}
\theoremstyle{plain}{\theorembodyfont{\rmfamily}%
}
\theoremstyle{plain}{\theorembodyfont{\rmfamily}%
\newtheorem{remark}[theorem]{Remark}}
\theoremstyle{plain}{\theorembodyfont{\rmfamily}%
}
\theoremstyle{plain}{\theorembodyfont{\rmfamily}%
}
\theoremstyle{plain}{\theorembodyfont{\rmfamily}%
}
\theoremstyle{plain}{\theorembodyfont{\rmfamily}%
\newtheorem{problem}[theorem]{Problem}}
\usepackage{enumitem}
\setlist[enumerate]{itemsep=-1pt}
\setlist[itemize]{itemsep=-1pt}

\usepackage{authblk}

\newcommand{\email}[1]{\href{mailto:#1}{\nolinkurl{#1}}}
\author{Minh N. B\`ui}
\affil{North Carolina State University,
Department of Mathematics,
Raleigh, NC 27695-8205, USA\\
\email{mnbui@ncsu.edu}
}
\begin{document}

\title{\sffamily\huge\vskip -10mm
Projective Splitting as a Warped Proximal Algorithm
}

\date{~}

\maketitle

\begin{abstract}
We show that the asynchronous block-iterative primal-dual
projective splitting framework introduced by P. L. Combettes and
J. Eckstein in their 2018 \emph{Math. Program.} paper can be
viewed as an instantiation of the recently proposed warped
proximal algorithm.
\end{abstract}


In \cite{Jmaa20}, the warped proximal algorithm was proposed and
its pertinence was illustrated through the ability to unify
existing methods such as those of
\cite{Siop14,Nfao15,Roc76a,Tsen00},
and to design novel flexible ones for solving
challenging monotone inclusions. Let us state a version of
\cite[Theorem~4.2]{Jmaa20}.

\begin{fact}
\label{f:1}
Let $\SHH$ be a real Hilbert space,
let $\boldsymbol{\mathsf{M}}\colon\SHH\to 2^{\SHH}$
be a maximally monotone operator such that
$\zer\boldsymbol{\mathsf{M}}\neq\emp$,
let $\boldsymbol{\mathsf{x}}_0\in\SHH$,
let $\varepsilon\in\zeroun$,
let $\alpha\in\RPP$, and let $\beta\in\left[\alpha,\pinf\right[$.
For every $n\in\NN$, let
$\boldsymbol{\mathsf{K}}_n\colon\SHH\to\SHH$
be $\alpha$-strongly monotone and $\beta$-Lipschitzian,
and let $\lambda_n\in\left[\varepsilon,2-\varepsilon\right]$.
Iterate
\begin{equation}
\label{e:1602}
\begin{array}{l}
\text{for}\;n=0,1,\ldots\\
\left\lfloor
\begin{array}{l}
\text{take}\;\widetilde{\boldsymbol{\mathsf{x}}}_n\in\SHH\\
\boldsymbol{\mathsf{y}}_n=
(\boldsymbol{\mathsf{K}}_n+\boldsymbol{\mathsf{M}})^{-1}
(\boldsymbol{\mathsf{K}}_n\widetilde{\boldsymbol{\mathsf{x}}}_n)\\
\boldsymbol{\mathsf{y}}_n^*=
\boldsymbol{\mathsf{K}}_n\widetilde{\boldsymbol{\mathsf{x}}}_n-
\boldsymbol{\mathsf{K}}_n\boldsymbol{\mathsf{y}}_n\\
\text{if}\;\scal{\boldsymbol{\mathsf{x}}_n-
\boldsymbol{\mathsf{y}}_n}{\boldsymbol{\mathsf{y}}_n^*}>0\\
\left\lfloor
\begin{array}{l}
\boldsymbol{\mathsf{x}}_{n+1}=
\boldsymbol{\mathsf{x}}_n-
\dfrac{\lambda_n\scal{\boldsymbol{\mathsf{x}}_n-
\boldsymbol{\mathsf{y}}_n}{\boldsymbol{\mathsf{y}}_n^*}}
{\|\boldsymbol{\mathsf{y}}_n^*\|^2}\,
\boldsymbol{\mathsf{y}}_n^*\\
\end{array}
\right.\\
\text{else}\\
\left\lfloor
\begin{array}{l}
\boldsymbol{\mathsf{x}}_{n+1}=\boldsymbol{\mathsf{x}}_n.
\end{array}
\right.\\[2mm]
\end{array}
\right.\\
\end{array}
\end{equation}
Then the following hold:
\begin{enumerate}
\item
\label{f:1i-}
$(\boldsymbol{\mathsf{x}}_n)_{n\in\NN}$ is bounded.
\item
\label{f:1i}
$\sum_{n\in\NN}\|\boldsymbol{\mathsf{x}}_{n+1}-
\boldsymbol{\mathsf{x}}_n\|^2<\pinf$.
\item
\label{f:1ii-}
$(\forall n\in\NN)$
$\scal{\boldsymbol{\mathsf{x}}_n-\boldsymbol{\mathsf{y}}_n}{
\boldsymbol{\mathsf{y}}_n^*}\leq
\varepsilon^{-1}\|\boldsymbol{\mathsf{y}}_n^*\|
\,\|\boldsymbol{\mathsf{x}}_{n+1}-\boldsymbol{\mathsf{x}}_n\|$.
\item
\label{f:1ii}
Suppose that $\widetilde{\boldsymbol{\mathsf{x}}}_n-
\boldsymbol{\mathsf{x}}_n\to\boldsymbol{\mathsf{0}}$.
Then $(\boldsymbol{\mathsf{x}}_n)_{n\in\NN}$ converges weakly to a
point in $\zer\boldsymbol{\mathsf{M}}$.
\end{enumerate}
\end{fact}

\newpage

A problem of interest in modern nonlinear analysis is the
following (see, e.g, \cite{Siop14,Moor21,MaPr18} and the
references therein for discussions on this problem).

\begin{problem}
\label{prob:1}
Let $(\HH_i)_{i\in I}$ and $(\GG_k)_{k\in K}$
be finite families of real Hilbert spaces.
For every $i\in I$ and every $k\in K$,
let $A_i\colon\HH_i\to 2^{\HH_i}$ and $B_k\colon\GG_k\to 2^{\GG_k}$
be maximally monotone,
let $z_i^*\in\HH_i$,
let $r_k\in\GG_k$, and
let $L_{k,i}\colon\HH_i\to\GG_k$ be linear and bounded.
The problem is to
\begin{equation}
\label{e:kt}
\text{find}\;\:(\overline{x}_i)_{i\in I}\in
\bigtimes_{i\in I}\HH_i
\;\:\text{and}\;\:
(\overline{v}_k^*)_{k\in K}\in
\bigtimes_{k\in K}\GG_k
\;\:\text{such that}\;\:
\begin{cases}
(\forall i\in I)\;\;z_i^*-
\Sum_{k\in K}L_{k,i}^*\overline{v}_k^*\in A_i\overline{x}_i\\
(\forall k\in K)\;\;
\Sum_{i\in I}L_{k,i}\overline{x}_i-r_k\in
B_k^{-1}\overline{v}_k^*.
\end{cases}
\end{equation}
The set of solutions to \eqref{e:kt} is denoted by
$\boldsymbol{\mathsf{Z}}$.
\end{problem}

The first asynchronous block-iterative algorithm to solve
Problem~\ref{prob:1} was proposed in \cite[Algorithm~12]{MaPr18}
as an extension of projective splitting techniques found in
\cite{Siop14,Ecks09}.
The present paper shows that \cite[Algorithm~12]{MaPr18}
can be viewed as a special case of \eqref{e:1602}.
Towards this goal, we first derive an abstract weak convergence
principle from Fact~\ref{f:1}. We refer the reader to
\cite{Livre1} for background in monotone operator theory and
nonlinear analysis.

\begin{theorem}
\label{t:1}
Let $\SHH$ be a real Hilbert space,
let $\boldsymbol{\mathsf{A}}\colon\SHH\to 2^{\SHH}$
be a maximally monotone operator, and let
$\boldsymbol{\mathsf{S}}\colon\SHH\to\SHH$ be a bounded linear
operator such that
$\boldsymbol{\mathsf{S}}^*={-}\boldsymbol{\mathsf{S}}$.
In addition, let $\boldsymbol{\mathsf{x}}_0\in\SHH$, 
let $\varepsilon\in\zeroun$,
let $\alpha\in\RPP$,
let $\rho\in\left[\alpha,\pinf\right[$,
and for every $n\in\NN$, let
$\boldsymbol{\mathsf{F}}_n\colon\SHH\to\SHH$ be
$\alpha$-strongly monotone and $\rho$-Lipschitzian,
and let $\lambda_n\in\left[\varepsilon,2-\varepsilon\right]$.
Iterate
\begin{equation}
\label{e:21}
\begin{array}{l}
\text{for}\;n=0,1,\ldots\\
\left\lfloor
\begin{array}{l}
\text{take}\;
\boldsymbol{\mathsf{u}}_n\in\SHH,\;
\boldsymbol{\mathsf{e}}_n^*\in\SHH,\;\text{and}\;
\boldsymbol{\mathsf{f}}_n^*\in\SHH\\
\boldsymbol{\mathsf{u}}_n^*=
\boldsymbol{\mathsf{F}}_n\boldsymbol{\mathsf{u}}_n-
\boldsymbol{\mathsf{S}}\boldsymbol{\mathsf{u}}_n+
\boldsymbol{\mathsf{e}}_n^*+
\boldsymbol{\mathsf{f}}_n^*
\\
\boldsymbol{\mathsf{y}}_n=
(\boldsymbol{\mathsf{F}}_n+\boldsymbol{\mathsf{A}})^{-1}
\boldsymbol{\mathsf{u}}_n^*\\
\boldsymbol{\mathsf{a}}_n^*=\boldsymbol{\mathsf{u}}_n^*-
\boldsymbol{\mathsf{F}}_n\boldsymbol{\mathsf{y}}_n\\
\boldsymbol{\mathsf{y}}_n^*=
\boldsymbol{\mathsf{a}}_n^*+
\boldsymbol{\mathsf{S}}\boldsymbol{\mathsf{y}}_n\\
\pi_n=\scal{\boldsymbol{\mathsf{x}}_n
}{\boldsymbol{\mathsf{y}}_n^*}-\scal{\boldsymbol{\mathsf{y}}_n
}{\boldsymbol{\mathsf{a}}_n^*}\\
\text{if}\;\pi_n>0\\
\left\lfloor
\begin{array}{l}
\tau_n=\|\boldsymbol{\mathsf{y}}_n^*\|^2\\
\theta_n=\lambda_n\pi_n/\tau_n\\
\boldsymbol{\mathsf{x}}_{n+1}=
\boldsymbol{\mathsf{x}}_n-\theta_n\boldsymbol{\mathsf{y}}_n^*
\end{array}
\right.
\\
\text{else}\\
\left\lfloor
\begin{array}{l}
\boldsymbol{\mathsf{x}}_{n+1}=\boldsymbol{\mathsf{x}}_n.
\end{array}
\right.
\end{array}
\right.
\end{array}
\end{equation}
Suppose that
$\zer(\boldsymbol{\mathsf{A}}+\boldsymbol{\mathsf{S}})\neq\emp$.
Then the following hold:
\begin{enumerate}
\item
\label{t:1a}
$\sum_{n\in\NN}\|\boldsymbol{\mathsf{x}}_{n+1}-
\boldsymbol{\mathsf{x}}_n\|^2<\pinf$.
\item
\label{t:1b}
Suppose that
$\boldsymbol{\mathsf{u}}_n-\boldsymbol{\mathsf{x}}_n
\to\boldsymbol{\mathsf{0}}$,
that $\boldsymbol{\mathsf{e}}_n^*\to\boldsymbol{\mathsf{0}}$,
that $(\boldsymbol{\mathsf{f}}_n^*)_{n\in\NN}$ is bounded,
and that there exists $\delta\in\zeroun$ such that
\begin{equation}
\label{e:1741}
(\forall n\in\NN)\quad
\begin{cases}
\scal{\boldsymbol{\mathsf{u}}_n
-\boldsymbol{\mathsf{y}}_n}{\boldsymbol{\mathsf{f}}_n^*}
\geq{-}\delta\scal{\boldsymbol{\mathsf{u}}_n-
\boldsymbol{\mathsf{y}}_n}{
\boldsymbol{\mathsf{F}}_n\boldsymbol{\mathsf{u}}_n-
\boldsymbol{\mathsf{F}}_n\boldsymbol{\mathsf{y}}_n}
\\
\scal{\boldsymbol{\mathsf{a}}_n^*+
\boldsymbol{\mathsf{S}}\boldsymbol{\mathsf{u}}_n-
\boldsymbol{\mathsf{e}}_n^*}{\boldsymbol{\mathsf{f}}_n^*}
\leq\delta\|\boldsymbol{\mathsf{a}}_n^*+
\boldsymbol{\mathsf{S}}\boldsymbol{\mathsf{u}}_n-
\boldsymbol{\mathsf{e}}_n^*\|^2.
\end{cases}
\end{equation}
Then $(\boldsymbol{\mathsf{x}}_n)_{n\in\NN}$ converges weakly to a
point in $\zer(\boldsymbol{\mathsf{A}}+\boldsymbol{\mathsf{S}})$.
\end{enumerate}
\end{theorem}
\begin{proof}
Set $\boldsymbol{\mathsf{M}}=
\boldsymbol{\mathsf{A}}+\boldsymbol{\mathsf{S}}$
and $(\forall n\in\NN)$ $\boldsymbol{\mathsf{K}}_n=
\boldsymbol{\mathsf{F}}_n-\boldsymbol{\mathsf{S}}$.
Then, it follows from
\cite[Example~20.35 and Corollary~25.5(i)]{Livre1}
that $\boldsymbol{\mathsf{M}}$ is maximally monotone
with $\zer\boldsymbol{\mathsf{M}}\neq\emp$.
Now take $n\in\NN$. We have
\begin{equation}
\label{e:5275}
\boldsymbol{\mathsf{K}}_n+\boldsymbol{\mathsf{M}}=
\boldsymbol{\mathsf{F}}_n+\boldsymbol{\mathsf{A}}.
\end{equation}
Since $\boldsymbol{\mathsf{S}}^*={-}\boldsymbol{\mathsf{S}}$,
we deduce that
\begin{equation}
\label{e:1225}
\text{$\boldsymbol{\mathsf{K}}_n$ is
$\alpha$-strongly monotone and $\beta$-Lipschitzian},
\end{equation}
where $\beta=\rho+\|\boldsymbol{\mathsf{S}}\|$. Thus,
\cite[Corollary~20.28 and Proposition~22.11(ii)]{Livre1}
guarantee that there exists
$\widetilde{\boldsymbol{\mathsf{x}}}_n\in\SHH$ such that
\begin{equation}
\label{e:9716}
\boldsymbol{\mathsf{u}}_n^*=
\boldsymbol{\mathsf{K}}_n\widetilde{\boldsymbol{\mathsf{x}}}_n.
\end{equation}
Hence, by \eqref{e:21} and \eqref{e:5275},
\begin{equation}
\label{e:2913}
\boldsymbol{\mathsf{y}}_n=
(\boldsymbol{\mathsf{K}}_n+\boldsymbol{\mathsf{M}})^{-1}
(\boldsymbol{\mathsf{K}}_n\widetilde{\boldsymbol{\mathsf{x}}}_n)
\quad\text{and}\quad
\boldsymbol{\mathsf{y}}_n^*=
\boldsymbol{\mathsf{u}}_n^*-
\boldsymbol{\mathsf{F}}_n\boldsymbol{\mathsf{y}}_n+
\boldsymbol{\mathsf{S}}\boldsymbol{\mathsf{y}}_n=
\boldsymbol{\mathsf{K}}_n\widetilde{\boldsymbol{\mathsf{x}}}_n-
\boldsymbol{\mathsf{K}}_n\boldsymbol{\mathsf{y}}_n.
\end{equation}
At the same time, we have
$\scal{\boldsymbol{\mathsf{y}}_n}{
\boldsymbol{\mathsf{S}}\boldsymbol{\mathsf{y}}_n}=0$
and it thus results from \eqref{e:21} that
$\pi_n=\scal{\boldsymbol{\mathsf{x}}_n
}{\boldsymbol{\mathsf{y}}_n^*}-\scal{\boldsymbol{\mathsf{y}}_n
}{\boldsymbol{\mathsf{a}}_n^*+
\boldsymbol{\mathsf{S}}\boldsymbol{\mathsf{y}}_n}
=\scal{\boldsymbol{\mathsf{x}}_n-\boldsymbol{\mathsf{y}}_n}{
\boldsymbol{\mathsf{y}}_n^*}$. Altogether,
\eqref{e:21} is a special case of \eqref{e:1602}.

\ref{t:1a}:
Fact~\ref{f:1}\ref{f:1i}.

\ref{t:1b}:
In the light of Fact~\ref{f:1}\ref{f:1ii},
it suffices to verify that
$\widetilde{\boldsymbol{\mathsf{x}}}_n-
\boldsymbol{\mathsf{x}}_n\to\boldsymbol{\mathsf{0}}$.
For every $n\in\NN$, since
$\boldsymbol{\mathsf{K}}_n+\boldsymbol{\mathsf{M}}$ is
maximally monotone \cite[Corollary~25.5(i)]{Livre1}
and $\alpha$-strongly monotone,
\cite[Example~22.7 and Proposition~22.11(ii)]{Livre1} implies that
$(\boldsymbol{\mathsf{K}}_n+\boldsymbol{\mathsf{M}})^{-1}
\colon\SHH\to\SHH$ is $(1/\alpha)$-Lipschitzian.
Therefore, we derive from \eqref{e:21}, \eqref{e:5275},
\cite[Proposition~3.10(i)]{Jmaa20}, and \eqref{e:1225} that
$(\forall\boldsymbol{\mathsf{z}}\in\zer\boldsymbol{\mathsf{M}})
(\forall n\in\NN)$
$\alpha\|\boldsymbol{\mathsf{y}}_n-\boldsymbol{\mathsf{z}}\|
=\alpha\|(\boldsymbol{\mathsf{K}}_n+
\boldsymbol{\mathsf{M}})^{-1}\boldsymbol{\mathsf{u}}_n^*-
(\boldsymbol{\mathsf{K}}_n+\boldsymbol{\mathsf{M}})^{-1}
(\boldsymbol{\mathsf{K}}_n\boldsymbol{\mathsf{z}})\|
\leq\|\boldsymbol{\mathsf{u}}_n^*-
\boldsymbol{\mathsf{K}}_n\boldsymbol{\mathsf{z}}\|
=\|\boldsymbol{\mathsf{K}}_n\boldsymbol{\mathsf{u}}_n
-\boldsymbol{\mathsf{K}}_n\boldsymbol{\mathsf{z}}
+\boldsymbol{\mathsf{e}}_n^*+\boldsymbol{\mathsf{f}}_n^*\|
\leq\|\boldsymbol{\mathsf{K}}_n\boldsymbol{\mathsf{u}}_n
-\boldsymbol{\mathsf{K}}_n\boldsymbol{\mathsf{z}}\|
+\|\boldsymbol{\mathsf{e}}_n^*\|+\|\boldsymbol{\mathsf{f}}_n^*\|
\leq\beta\|\boldsymbol{\mathsf{u}}_n-\boldsymbol{\mathsf{z}}\|
+\|\boldsymbol{\mathsf{e}}_n^*\|+\|\boldsymbol{\mathsf{f}}_n^*\|$.
Thus, since Fact~\ref{f:1}\ref{f:1i-} and our assumption imply that
$(\boldsymbol{\mathsf{u}}_n)_{n\in\NN}$ is bounded, it follows
that $(\boldsymbol{\mathsf{y}}_n)_{n\in\NN}$ is bounded.
At the same time, for every $n\in\NN$,
we get from \eqref{e:21} that
\begin{equation}
\label{e:8157}
\boldsymbol{\mathsf{y}}_n^*
=\boldsymbol{\mathsf{F}}_n\boldsymbol{\mathsf{u}}_n
-\boldsymbol{\mathsf{F}}_n\boldsymbol{\mathsf{y}}_n
+\boldsymbol{\mathsf{e}}_n^*+\boldsymbol{\mathsf{f}}_n^*
-(\boldsymbol{\mathsf{S}}\boldsymbol{\mathsf{u}}_n
-\boldsymbol{\mathsf{S}}\boldsymbol{\mathsf{y}}_n)
=\boldsymbol{\mathsf{K}}_n\boldsymbol{\mathsf{u}}_n
-\boldsymbol{\mathsf{K}}_n\boldsymbol{\mathsf{y}}_n
+\boldsymbol{\mathsf{e}}_n^*+\boldsymbol{\mathsf{f}}_n^*
\end{equation}
and, thus, from \eqref{e:1225} that
$\|\boldsymbol{\mathsf{y}}_n^*\|
\leq\|\boldsymbol{\mathsf{K}}_n\boldsymbol{\mathsf{u}}_n
-\boldsymbol{\mathsf{K}}_n\boldsymbol{\mathsf{y}}_n\|+
\|\boldsymbol{\mathsf{e}}_n^*\|+\|\boldsymbol{\mathsf{f}}_n^*\|
\leq\beta\|\boldsymbol{\mathsf{u}}_n-\boldsymbol{\mathsf{y}}_n\|+
\|\boldsymbol{\mathsf{e}}_n^*\|+\|\boldsymbol{\mathsf{f}}_n^*\|$.
Thus, $(\boldsymbol{\mathsf{y}}_n^*)_{n\in\NN}$
is bounded, from which, \ref{t:1a}, and
Fact~\ref{f:1}\ref{f:1ii-} we obtain
$\varlimsup
\scal{\boldsymbol{\mathsf{x}}_n-\boldsymbol{\mathsf{y}}_n}{
\boldsymbol{\mathsf{y}}_n^*}\leq 0$.
In turn, since $\boldsymbol{\mathsf{x}}_n-
\boldsymbol{\mathsf{u}}_n\to\boldsymbol{\mathsf{0}}$
and $\boldsymbol{\mathsf{e}}_n^*\to\boldsymbol{\mathsf{0}}$,
it results from \eqref{e:8157} and \eqref{e:1741} that
\begin{align}
0&\geq\varlimsup
\scal{\boldsymbol{\mathsf{x}}_n-\boldsymbol{\mathsf{y}}_n}{
\boldsymbol{\mathsf{y}}_n^*}
\nonumber\\
&=\varlimsup\big(
\scal{\boldsymbol{\mathsf{u}}_n-\boldsymbol{\mathsf{y}}_n}{
\boldsymbol{\mathsf{y}}_n^*}+
\scal{\boldsymbol{\mathsf{x}}_n-\boldsymbol{\mathsf{u}}_n}{
\boldsymbol{\mathsf{y}}_n^*}\big)
\nonumber\\
&=\varlimsup\scal{\boldsymbol{\mathsf{u}}_n-
\boldsymbol{\mathsf{y}}_n}{\boldsymbol{\mathsf{y}}_n^*}
\nonumber\\
&=\varlimsup\big(\scal{\boldsymbol{\mathsf{u}}_n-
\boldsymbol{\mathsf{y}}_n}{
\boldsymbol{\mathsf{F}}_n\boldsymbol{\mathsf{u}}_n-
\boldsymbol{\mathsf{F}}_n\boldsymbol{\mathsf{y}}_n+
\boldsymbol{\mathsf{e}}_n^*+\boldsymbol{\mathsf{f}}_n^*}-
\scal{\boldsymbol{\mathsf{u}}_n-
\boldsymbol{\mathsf{y}}_n}{
\boldsymbol{\mathsf{S}}\boldsymbol{\mathsf{u}}_n-
\boldsymbol{\mathsf{S}}\boldsymbol{\mathsf{y}}_n
}\big)
\nonumber\\
&=\varlimsup\big(\scal{\boldsymbol{\mathsf{u}}_n-
\boldsymbol{\mathsf{y}}_n}{
\boldsymbol{\mathsf{F}}_n\boldsymbol{\mathsf{u}}_n-
\boldsymbol{\mathsf{F}}_n\boldsymbol{\mathsf{y}}_n+
\boldsymbol{\mathsf{f}}_n^*}+
\scal{\boldsymbol{\mathsf{u}}_n-
\boldsymbol{\mathsf{y}}_n}{\boldsymbol{\mathsf{e}}_n^*}
\big)
\nonumber\\
&\geq\varlimsup\big((1-\delta)
\scal{\boldsymbol{\mathsf{u}}_n-
\boldsymbol{\mathsf{y}}_n}{
\boldsymbol{\mathsf{F}}_n\boldsymbol{\mathsf{u}}_n-
\boldsymbol{\mathsf{F}}_n\boldsymbol{\mathsf{y}}_n}+
\scal{\boldsymbol{\mathsf{u}}_n-
\boldsymbol{\mathsf{y}}_n}{\boldsymbol{\mathsf{e}}_n^*}
\big)
\nonumber\\
&\geq\varlimsup\alpha(1-\delta)\|\boldsymbol{\mathsf{u}}_n-
\boldsymbol{\mathsf{y}}_n\|^2
\nonumber\\
&\geq\varlimsup\alpha(1-\delta)\rho^{-2}
\|\boldsymbol{\mathsf{F}}_n\boldsymbol{\mathsf{u}}_n-
\boldsymbol{\mathsf{F}}_n\boldsymbol{\mathsf{y}}_n\|^2.
\end{align}
Hence, $\boldsymbol{\mathsf{F}}_n\boldsymbol{\mathsf{u}}_n-
\boldsymbol{\mathsf{F}}_n\boldsymbol{\mathsf{y}}_n\to
\boldsymbol{\mathsf{0}}$. On the other hand, since
$(\boldsymbol{\mathsf{f}}_n^*)_{n\in\NN}$ is bounded and since
\eqref{e:21} yields
$(\boldsymbol{\mathsf{a}}_n^*+
\boldsymbol{\mathsf{S}}\boldsymbol{\mathsf{u}}_n-
\boldsymbol{\mathsf{e}}_n^*)_{n\in\NN}=
(\boldsymbol{\mathsf{F}}_n\boldsymbol{\mathsf{u}}_n-
\boldsymbol{\mathsf{F}}_n\boldsymbol{\mathsf{y}}_n+
\boldsymbol{\mathsf{f}}_n^*)_{n\in\NN}$,
we derive from \eqref{e:1741} that
\begin{align}
\varlimsup(1-\delta)\|\boldsymbol{\mathsf{f}}_n^*\|^2
&=\varlimsup\big(
\scal{\boldsymbol{\mathsf{F}}_n\boldsymbol{\mathsf{u}}_n-
\boldsymbol{\mathsf{F}}_n\boldsymbol{\mathsf{y}}_n}{
\boldsymbol{\mathsf{f}}_n^*}+
(1-\delta)\|\boldsymbol{\mathsf{f}}_n^*\|^2\big)
\nonumber\\
&=\varlimsup\big(
\scal{\boldsymbol{\mathsf{F}}_n\boldsymbol{\mathsf{u}}_n-
\boldsymbol{\mathsf{F}}_n\boldsymbol{\mathsf{y}}_n+
\boldsymbol{\mathsf{f}}_n^*}{
\boldsymbol{\mathsf{f}}_n^*}-
\delta\|\boldsymbol{\mathsf{f}}_n^*\|^2\big)
\nonumber\\
&\leq\varlimsup\big(
\delta\|\boldsymbol{\mathsf{F}}_n\boldsymbol{\mathsf{u}}_n-
\boldsymbol{\mathsf{F}}_n\boldsymbol{\mathsf{y}}_n+
\boldsymbol{\mathsf{f}}_n^*\|^2-
\delta\|\boldsymbol{\mathsf{f}}_n^*\|^2\big)
\nonumber\\
&=\varlimsup\big(
\delta\|\boldsymbol{\mathsf{F}}_n\boldsymbol{\mathsf{u}}_n-
\boldsymbol{\mathsf{F}}_n\boldsymbol{\mathsf{y}}_n\|^2
+2\delta\scal{\boldsymbol{\mathsf{F}}_n\boldsymbol{\mathsf{u}}_n-
\boldsymbol{\mathsf{F}}_n\boldsymbol{\mathsf{y}}_n}{
\boldsymbol{\mathsf{f}}_n^*}\big)
\nonumber\\
&=0.
\end{align}
Therefore,
$\boldsymbol{\mathsf{f}}_n^*\to\boldsymbol{\mathsf{0}}$.
Consequently, by \eqref{e:1225}, \eqref{e:9716}, and \eqref{e:21},
$\alpha\|\widetilde{\boldsymbol{\mathsf{x}}}_n-
\boldsymbol{\mathsf{x}}_n\|\leq
\|\boldsymbol{\mathsf{K}}_n\widetilde{\boldsymbol{\mathsf{x}}}_n-
\boldsymbol{\mathsf{K}}_n\boldsymbol{\mathsf{x}}_n\|
=\|\boldsymbol{\mathsf{K}}_n\boldsymbol{\mathsf{u}}_n-
\boldsymbol{\mathsf{K}}_n\boldsymbol{\mathsf{x}}_n
+\boldsymbol{\mathsf{e}}_n^*+\boldsymbol{\mathsf{f}}_n^*\|
\leq\beta\|\boldsymbol{\mathsf{u}}_n-\boldsymbol{\mathsf{x}}_n\|+
\|\boldsymbol{\mathsf{e}}_n^*\|+\|\boldsymbol{\mathsf{f}}_n^*\|
\to 0$.
\end{proof}

We are now ready to recover \cite[Theorem~13]{MaPr18}. Recall
that, given a real Hilbert space $\HH$ with identity operator
$\Id$, the resolvent of an operator $A\colon\HH\to 2^\HH$ is
$J_A=(\Id+A)^{-1}$.

\begin{corollary}[\cite{MaPr18}]
\label{c:1}
Consider the setting of Problem~\ref{prob:1}
and suppose that $\boldsymbol{\mathsf{Z}}\neq\emp$.
Let $(I_n)_{n\in\NN}$ be nonempty subsets of $I$
and $(K_n)_{n\in\NN}$ be nonempty subsets of $K$ such that
\begin{equation}
\label{e:4503}
I_0=I,\quad
K_0=K,\quad\text{and}\quad
(\exi T\in\NN)(\forall n\in\NN)\;\;
\bigcup_{j=n}^{n+T}I_j=I\;\:\text{and}\;\:
\bigcup_{j=n}^{n+T}K_j=K.
\end{equation}
In addition, let $D\in\NN$, let $\varepsilon\in\zeroun$,
let $(\lambda_n)_{n\in\NN}$ be in
$\left[\varepsilon,2-\varepsilon\right]$,
and for every $i\in I$ and every $k\in K$,
let $(c_i(n))_{n\in\NN}$ and $(d_k(n))_{n\in\NN}$ be in $\NN$
such that
\begin{equation}
\label{e:5550}
(\forall n\in\NN)\quad
n-D\leq c_i(n)\leq n\quad\text{and}\quad
n-D\leq d_k(n)\leq n,
\end{equation}
let $(\gamma_{i,n})_{n\in\NN}$ and $(\mu_{k,n})_{n\in\NN}$
be in $\left[\varepsilon,1/\varepsilon\right]$,
let $x_{i,0}\in\HH_i$, and let $v_{k,0}^*\in\GG_k$. Iterate
\begin{equation}
\label{e:1852}
\begin{array}{l}
\text{for}\;n=0,1,\ldots\\
\left\lfloor
\begin{array}{l}
\text{for every}\;i\in I_n\\
\left\lfloor
\begin{array}{l}
\text{take}\;e_{i,n}\in\HH_i\\
l_{i,n}^*=\sum_{k\in K}L_{k,i}^*v_{k,c_i(n)}^*\\
a_{i,n}=J_{\gamma_{i,c_i(n)}A_i}\big(x_{i,c_i(n)}+
\gamma_{i,c_i(n)}(z_i^*-l_{i,n}^*)+e_{i,n}\big)\\
a_{i,n}^*=\gamma_{i,c_i(n)}^{-1}(x_{i,c_i(n)}-a_{i,n}+e_{i,n})
-l_{i,n}^*
\end{array}
\right.
\\
\text{for every}\;i\in I\smallsetminus I_n\\
\left\lfloor
\begin{array}{l}
a_{i,n}=a_{i,n-1}\\
a_{i,n}^*=a_{i,n-1}^*
\end{array}
\right.
\\
\text{for every}\;k\in K_n\\
\left\lfloor
\begin{array}{l}
\text{take}\;f_{k,n}\in\GG_k\\
l_{k,n}=\sum_{i\in I}L_{k,i}x_{i,d_k(n)}\\
b_{k,n}=r_k+J_{\mu_{k,d_k(n)}B_k}\big(l_{k,n}
+\mu_{k,d_k(n)}v_{k,d_k(n)}^*+f_{k,n}-r_k\big)\\
b_{k,n}^*=v_{k,d_k(n)}^*+\mu_{k,d_k(n)}^{-1}(l_{k,n}-b_{k,n}
+f_{k,n})\\
t_{k,n}=b_{k,n}-\sum_{i\in I}L_{k,i}a_{i,n}
\end{array}
\right.
\\
\text{for every}\;k\in K\smallsetminus K_n\\
\left\lfloor
\begin{array}{l}
b_{k,n}=b_{k,n-1}\\
b_{k,n}^*=b_{k,n-1}^*\\
t_{k,n}=b_{k,n}-\sum_{i\in I}L_{k,i}a_{i,n}
\end{array}
\right.
\\
\text{for every}\;i\in I\\
\left\lfloor
\begin{array}{l}
t_{i,n}^*=a_{i,n}^*+\sum_{k\in K}L_{k,i}^*b_{k,n}^*
\end{array}
\right.
\\
\pi_n=\sum_{i\in I}\big(\scal{x_{i,n}}{t_{i,n}^*}
-\scal{a_{i,n}}{a_{i,n}^*}\big)+
\sum_{k\in K}\big(\scal{t_{k,n}}{v_{k,n}^*}-
\scal{b_{k,n}}{b_{k,n}^*}\big)\\
\text{if}\;\pi_n>0\\
\left\lfloor
\begin{array}{l}
\tau_n=\sum_{i\in I}\|t_{i,n}^*\|^2+\sum_{k\in K}\|t_{k,n}\|^2\\
\theta_n=\lambda_n\pi_n/\tau_n
\end{array}
\right.
\\
\text{else}\\
\left\lfloor
\begin{array}{l}
\theta_n=0
\end{array}
\right.
\\
\text{for every}\;i\in I\\
\left\lfloor
\begin{array}{l}
x_{i,n+1}=x_{i,n}-\theta_nt_{i,n}^*
\end{array}
\right.
\\
\text{for every}\;k\in K\\
\left\lfloor
\begin{array}{l}
v_{k,n+1}^*=v_{k,n}^*-\theta_nt_{k,n}.
\end{array}
\right.\\[3mm]
\end{array}
\right.
\end{array}
\end{equation}
In addition, suppose that there exist
$\eta\in\RPP$, $\chi\in\RPP$, $\sigma\in\zeroun$, and
$\zeta\in\zeroun$ such that
\begin{equation}
\label{e:606a}
(\forall n\in\NN)(\forall i\in I_n)\quad
\begin{cases}
\|e_{i,n}\|\leq\eta\\
\scal{x_{i,c_i(n)}-a_{i,n}}{e_{i,n}}
\geq{-}\sigma\|x_{i,c_i(n)}-a_{i,n}\|^2\\
\scal{e_{i,n}}{a_{i,n}^*+l_{i,n}^*}\leq\sigma\gamma_{i,c_i(n)}
\|a_{i,n}^*+l_{i,n}^*\|^2
\end{cases}
\end{equation}
and that
\begin{equation}
\label{e:606b}
(\forall n\in\NN)(\forall k\in K_n)\quad
\begin{cases}
\|f_{k,n}\|\leq\chi\\
\scal{l_{k,n}-b_{k,n}}{f_{k,n}}\geq{-}\zeta\|l_{k,n}-b_{k,n}\|^2
\\
\scal{f_{k,n}}{b_{k,n}^*-v_{k,d_k(n)}^*}
\leq\zeta\mu_{k,d_k(n)}\|b_{k,n}^*-v_{k,d_k(n)}^*\|^2.
\end{cases}
\end{equation}
Then $((x_{i,n})_{i\in I},(v_{k,n}^*)_{k\in K})_{n\in\NN}$
converges weakly to a point in $\boldsymbol{\mathsf{Z}}$.
\end{corollary}
\begin{proof}
Denote by $\HHH$ and $\GGG$ the Hilbert direct sums
of $(\HH_i)_{i\in I}$ and $(\GG_k)_{k\in K}$,
set $\SHH=\HHH\oplus\GGG$, and define the operators
\begin{equation}
\label{e:894a}
\boldsymbol{\mathsf{A}}\colon\SHH\to 2^{\SHH}\colon
\big((x_i)_{i\in I},(v_k^*)_{k\in K}\big)\mapsto
\bigg(\bigtimes_{i\in I}\big({-}z_i^*+A_ix_i\big)\bigg)
\times\bigg(\bigtimes_{k\in K}\big(r_k+B_k^{-1}v_k^*\big)\bigg)
\end{equation}
and
\begin{equation}
\label{e:894s}
\boldsymbol{\mathsf{S}}\colon\SHH\to\SHH\colon
\big((x_i)_{i\in I},(v_k^*)_{k\in K}\big)\mapsto\Bigg(
\Bigg(\sum_{k\in K}L_{k,i}^*v_k^*\Bigg)_{i\in I},
\Bigg({-}\sum_{i\in I}L_{k,i}x_i\Bigg)_{k\in K}\Bigg).
\end{equation}
Using the maximal monotonicity of the operators
$(A_i)_{i\in I}$ and $(B_k)_{k\in K}$,
we deduce from \cite[Proposition~20.23]{Livre1} that
$\boldsymbol{\mathsf{A}}$ is maximally monotone.
In addition, we observe that $\boldsymbol{\mathsf{S}}$ is a
bounded linear operator with
$\boldsymbol{\mathsf{S}}^*={-}\boldsymbol{\mathsf{S}}$.
At the same time, it results from \eqref{e:894a}, \eqref{e:894s},
and \eqref{e:kt} that
\begin{equation}
\label{e:7968}
\zer(\boldsymbol{\mathsf{A}}+\boldsymbol{\mathsf{S}})=
\boldsymbol{\mathsf{Z}}\neq\emp.
\end{equation}
Furthermore, \eqref{e:1852} yields
\begin{equation}
\label{e:4397}
\big[\;
(\forall i\in I)(\forall n\in\NN)\;\;
a_{i,n}^*\in{-}z_i^*+A_ia_{i,n}\;\big]
\quad\text{and}\quad
\big[\;(\forall k\in K)(\forall n\in\NN)\;\;
b_{k,n}\in r_k+B_k^{-1}b_{k,n}^*
\;\big].
\end{equation}
Next, define
\begin{equation}
\label{e:6052}
(\forall i\in I)(\forall n\in\NN)\quad
\begin{cases}
\overline{\ell}_i(n)=\max\menge{j\in\NN}{j\leq n\;\:
\text{and}\;\:i\in I_j},\;\:
\ell_i(n)=c_i\big(\overline{\ell}_i(n)\big)\\
u_{i,n}^*=\gamma_{i,\ell_i(n)}^{-1}x_{i,\ell_i(n)}-
l_{i,\overline{\ell}_i(n)}^*+
\gamma_{i,\ell_i(n)}^{-1}e_{i,\overline{\ell}_i(n)}\\
w_{i,n}^*=\sum_{k\in K}L_{k,i}^*v_{k,\vartheta_k(n)}^*-
l_{i,\overline{\ell}_i(n)}^*.
\end{cases}
\end{equation}
Then, for every $i\in I$ and every $n\in\NN$, it follows from
\eqref{e:1852} that
\begin{equation}
\label{e:989a}
a_{i,n}
=a_{i,\overline{\ell}_i(n)}
=J_{\gamma_{i,\ell_i(n)}A_i}\big(\gamma_{i,\ell_i(n)}(u_{i,n}^*+
z_i^*)\big)
=\big(\gamma_{i,\ell_i(n)}^{-1}\Id-z_i^*+A_i\big)^{-1}u_{i,n}^*
\end{equation}
and, therefore, that
\begin{equation}
\label{e:989b}
a_{i,n}^*
=a_{i,\overline{\ell}_i(n)}^*
=u_{i,n}^*-\gamma_{i,\ell_i(n)}^{-1}a_{i,\overline{\ell}_i(n)}
=u_{i,n}^*-\gamma_{i,\ell_i(n)}^{-1}a_{i,n}.
\end{equation}
Likewise, for every $k\in K$ and every $n\in\NN$, upon setting
\begin{equation}
\label{e:3599}
\begin{cases}
\overline{\vartheta}_k(n)=\max\menge{j\in\NN}{j\leq n\;\:
\text{and}\;\:k\in K_j},\;\:
\vartheta_k(n)=d_k\big(\overline{\vartheta}_k(n)\big)\\
v_{k,n}=\mu_{k,\vartheta_k(n)}v_{k,\vartheta_k(n)}^*+
l_{k,\overline{\vartheta}_k(n)}+
f_{k,\overline{\vartheta}_k(n)}
\\
w_{k,n}=l_{k,\overline{\vartheta}_k(n)}-
\sum_{i\in I}L_{k,i}x_{i,\ell_i(n)},
\end{cases}
\end{equation}
we get from \eqref{e:1852} and
\cite[Proposition~23.17(iii)]{Livre1} that
\begin{equation}
\label{e:479a}
b_{k,n}
=b_{k,\overline{\vartheta}_k(n)}
=J_{\mu_{k,\vartheta_k(n)}B_k(\mute-r_k)}v_{k,n}
\end{equation}
and, in turn, from
\eqref{e:1852} and \cite[Proposition~23.20]{Livre1} that
\begin{align}
b_{k,n}^*
&=b_{k,\overline{\vartheta}_k(n)}^*
\label{e:2005}\\
&=\mu_{k,\vartheta_k(n)}^{-1}\big(v_{k,n}-
b_{k,\overline{\vartheta}_k(n)}\big)
\nonumber\\
&=\mu_{k,\vartheta_k(n)}^{-1}(v_{k,n}-b_{k,n})
\label{e:479b}
\\
&=J_{\mu_{k,\vartheta_k(n)}^{-1}(r_k+B_k^{-1})}
\big(\mu_{k,\vartheta_k(n)}^{-1}v_{k,n}\big)
\nonumber\\
&=\big(\mu_{k,\vartheta_k(n)}\Id+r_k+B_k^{-1}\big)^{-1}v_{k,n}.
\label{e:479c}
\end{align}
Let us set
\begin{equation}
\label{e:9986}
(\forall n\in\NN)\quad
\begin{cases}
\boldsymbol{\mathsf{x}}_n
=\big((x_{i,n})_{i\in I},(v_{k,n}^*)_{k\in K}\big),
\;\:
\boldsymbol{\mathsf{u}}_n
=\big(\big(x_{i,\ell_i(n)}\big)_{i\in I},
\big(v_{k,\vartheta_k(n)}^*\big)_{k\in K}\big)\\
\boldsymbol{\mathsf{e}}_n^*
=\big((w_{i,n}^*)_{i\in I},(w_{k,n})_{k\in K}\big),
\;\:
\boldsymbol{\mathsf{f}}_n^*
=\big(\big(\gamma_{i,\ell_i(n)}^{-1}
e_{i,\overline{\ell}_i(n)}\big)_{i\in I},
\big(f_{k,\overline{\vartheta}_k(n)}\big)_{k\in K}\big)\\
\boldsymbol{\mathsf{u}}_n^*
=\big((u_{i,n}^*)_{i\in I},(v_{k,n})_{k\in K}\big),
\;\:
\boldsymbol{\mathsf{y}}_n
=\big((a_{i,n})_{i\in I},(b_{k,n}^*)_{k\in K}\big)\\
\boldsymbol{\mathsf{a}}_n^*=
\big((a_{i,n}^*)_{i\in I},(b_{k,n})_{k\in K}\big),
\;\:
\boldsymbol{\mathsf{y}}_n^*=
\big((t_{i,n}^*)_{i\in I},(t_{k,n})_{k\in K}\big)\\
\boldsymbol{\mathsf{F}}_n\colon\SHH\to\SHH\colon
\big((x_i)_{i\in I},(v_k^*)_{k\in K}\big)\mapsto\big(
\big(\gamma_{i,\ell_i(n)}^{-1}x_i\big)_{i\in I},
\big(\mu_{k,\vartheta_k(n)}v_k^*\big)_{k\in K}\big).
\end{cases}
\end{equation}
Then, the operators $(\boldsymbol{\mathsf{F}}_n)_{n\in\NN}$
are $\varepsilon$-strongly monotone and
$(1/\varepsilon)$-Lipschitzian.
For every $n\in\NN$, by virtue of \eqref{e:6052} and
\eqref{e:3599}, we deduce from \eqref{e:894s} that
\begin{equation}
\label{e:2156}
\boldsymbol{\mathsf{S}}\boldsymbol{\mathsf{u}}_n-
\boldsymbol{\mathsf{e}}_n^*
=\Big(\big(l_{i,\overline{\ell}_i(n)}^*\big)_{i\in I},
\big({-}l_{k,\overline{\vartheta}_k(n)}\big)_{k\in K}\Big),
\end{equation}
which yields
\begin{equation}
\label{e:5929}
\boldsymbol{\mathsf{u}}_n^*=
\boldsymbol{\mathsf{F}}_n\boldsymbol{\mathsf{u}}_n-
\boldsymbol{\mathsf{S}}\boldsymbol{\mathsf{u}}_n+
\boldsymbol{\mathsf{e}}_n^*+
\boldsymbol{\mathsf{f}}_n^*.
\end{equation}
Furthermore, we infer from \eqref{e:989a}, \eqref{e:479c},
and \eqref{e:894a} that
\begin{equation}
\label{e:5939}
(\forall n\in\NN)\quad
\boldsymbol{\mathsf{y}}_n
=(\boldsymbol{\mathsf{F}}_n+\boldsymbol{\mathsf{A}})^{-1}
\boldsymbol{\mathsf{u}}_n^*.
\end{equation}
At the same time, \eqref{e:989b} and \eqref{e:479b} imply that
\begin{equation}
\label{e:3823}
(\forall n\in\NN)\quad
\boldsymbol{\mathsf{a}}_n^*=
\boldsymbol{\mathsf{u}}_n^*-
\boldsymbol{\mathsf{F}}_n\boldsymbol{\mathsf{y}}_n,
\end{equation}
while \eqref{e:9986}, \eqref{e:1852}, and
\eqref{e:894s} guarantee that
\begin{equation}
\label{e:5622}
(\forall n\in\NN)\quad\boldsymbol{\mathsf{y}}_n^*
=\boldsymbol{\mathsf{a}}_n^*+
\boldsymbol{\mathsf{S}}\boldsymbol{\mathsf{y}}_n
\quad\text{and}\quad
\pi_n=\scal{\boldsymbol{\mathsf{x}}_n
}{\boldsymbol{\mathsf{y}}_n^*}-\scal{\boldsymbol{\mathsf{y}}_n
}{\boldsymbol{\mathsf{a}}_n^*}.
\end{equation}
Altogether, it follows from \eqref{e:5929}--\eqref{e:5622}
that \eqref{e:1852} is an instantiation of \eqref{e:21}.
Hence, Theorem~\ref{t:1}\ref{t:1a} yields
$\sum_{n\in\NN}\|\boldsymbol{\mathsf{x}}_{n+1}-
\boldsymbol{\mathsf{x}}_n\|^2<\pinf$.
In turn, using \eqref{e:4503}, \eqref{e:5550},
\eqref{e:6052}, and \eqref{e:3599}, we deduce from
\cite[Lemma~A.3]{Moor21} that, for every $i\in I$
and every $k\in K$, we have
$\boldsymbol{\mathsf{x}}_{\ell_i(n)}-\boldsymbol{\mathsf{x}}_n
\to\boldsymbol{\mathsf{0}}$ and 
$\boldsymbol{\mathsf{x}}_{\vartheta_k(n)}-\boldsymbol{\mathsf{x}}_n
\to\boldsymbol{\mathsf{0}}$.
This and \eqref{e:9986} imply that
\begin{equation}
\label{e:2456}
\boldsymbol{\mathsf{u}}_n-\boldsymbol{\mathsf{x}}_n\to
\boldsymbol{\mathsf{0}}.
\end{equation}
Moreover, we deduce from \eqref{e:6052} that
\begin{equation}
(\forall i\in I)\quad
\|w_{i,n}^*\|\leq\sum_{k\in K}\|L_{k,i}^*\|\,
\big\|v_{k,\vartheta_k(n)}^*-v_{k,\ell_i(n)}^*\big\|
\leq\sum_{k\in K}\|L_{k,i}^*\|\,\|
\boldsymbol{\mathsf{x}}_{\vartheta_k(n)}-
\boldsymbol{\mathsf{x}}_{\ell_i(n)}\|\to 0
\end{equation}
and from \eqref{e:3599} that
\begin{equation}
(\forall k\in K)\quad
\|w_{k,n}\|\leq\sum_{i\in I}\|L_{k,i}\|\,\|
x_{i,\vartheta_k(n)}-x_{i,\ell_i(n)}\|
\leq\sum_{i\in I}\|L_{k,i}\|\,\|
\boldsymbol{\mathsf{x}}_{\vartheta_k(n)}-
\boldsymbol{\mathsf{x}}_{\ell_i(n)}\|\to 0.
\end{equation}
Therefore,
$\boldsymbol{\mathsf{e}}_n^*\to\boldsymbol{\mathsf{0}}$.
By \eqref{e:606a} and \eqref{e:606b},
$(\boldsymbol{\mathsf{f}}_n^*)_{n\in\NN}$ is bounded.
In view of \eqref{e:9986}, \eqref{e:606a}, and \eqref{e:606b},
we get from \eqref{e:989a} and \eqref{e:2005} that
\begin{align}
(\forall n\in\NN)\quad
\scal{\boldsymbol{\mathsf{u}}_n
-\boldsymbol{\mathsf{y}}_n}{\boldsymbol{\mathsf{f}}_n^*}
&=\sum_{i\in I}\sscal{x_{i,\ell_i(n)}
-a_{i,n}}{\gamma_{i,\ell_i(n)}^{-1}e_{i,\overline{\ell}_i(n)}}
+\sum_{k\in K}\sscal{v_{k,\vartheta_k(n)}^*-b_{k,n}^*}{
f_{k,\overline{\vartheta}_k(n)}}
\nonumber\\
&\geq{-}\sigma\sum_{i\in I}\gamma_{i,\ell_i(n)}^{-1}\|
x_{i,\ell_i(n)}-a_{i,n}\|^2
-\zeta\sum_{k\in K}\mu_{k,\vartheta_k(n)}
\|v_{k,\vartheta_k(n)}^*-b_{k,n}^*\|^2
\nonumber\\
&\geq{-}\max\{\sigma,\zeta\}\scal{\boldsymbol{\mathsf{u}}_n
-\boldsymbol{\mathsf{y}}_n}{
\boldsymbol{\mathsf{F}}_n\boldsymbol{\mathsf{u}}_n-
\boldsymbol{\mathsf{F}}_n\boldsymbol{\mathsf{y}}_n},
\end{align}
and from \eqref{e:2156}, \eqref{e:989b}, and \eqref{e:479a} that
\begin{align}
\label{e:1184}
\scal{\boldsymbol{\mathsf{a}}_n^*+
\boldsymbol{\mathsf{S}}\boldsymbol{\mathsf{u}}_n-
\boldsymbol{\mathsf{e}}_n^*}
{\boldsymbol{\mathsf{f}}_n^*}
&=\sum_{i\in I}\sscal{a_{i,\overline{\ell}_i(n)}^*+
l_{i,\overline{\ell}_i(n)}^*
}{\gamma_{i,\ell_i(n)}^{-1}
e_{i,\overline{\ell}_i(n)}}
+\sum_{k\in K}\sscal{b_{k,\overline{\vartheta}_k(n)}-
l_{k,\overline{\vartheta}_k(n)}}{
f_{k,\overline{\vartheta}_k(n)}}
\nonumber\\
&\leq\sigma\sum_{i\in I}
\big\|a_{i,\overline{\ell}_i(n)}^*+
l_{i,\overline{\ell}_i(n)}^*\big\|^2+
\zeta\sum_{k\in K}\big\|b_{k,\overline{\vartheta}_k(n)}-
l_{k,\overline{\vartheta}_k(n)}\big\|^2
\nonumber\\
&\leq\max\{\sigma,\zeta\}\|\boldsymbol{\mathsf{a}}_n^*+
\boldsymbol{\mathsf{S}}\boldsymbol{\mathsf{u}}_n-
\boldsymbol{\mathsf{e}}_n^*\|^2.
\end{align}
Altogether, the conclusion follows from
Theorem~\ref{t:1}\ref{t:1b}.
\end{proof}

\begin{remark}
\label{r:1}
Using similar arguments, one can show that the
asynchronous strongly convergent block-iterative method
\cite[Algorithm~14]{MaPr18} can be viewed as an instance of
\cite[Theorem~4.8]{Jmaa20}.
\end{remark}

\begin{ack}
This work is a part of the author's Ph.D. dissertation and it
was supported by the National Science Foundation under
grant DMS-1818946. The author thanks his Ph.D. advisor P. L.
Combettes for his guidance during this work.
\end{ack}

\end{document}